%% file: compile.tex
\documentclass[11pt, reqno]{amsart}

\usepackage{amsmath}
\usepackage[margin=1.00in]{geometry}
\usepackage[foot]{amsaddr}

\usepackage{amssymb}
\usepackage{physics}
\usepackage{mathrsfs}
\usepackage{bbold}

\usepackage{enumerate}         

\usepackage{color}
\usepackage{xcolor}
\usepackage{stmaryrd}
\usepackage{cases}
\usepackage{url} 
\usepackage{amsthm}
\usepackage{hyperref}
\usepackage{bm}
\usepackage{xy}
\usepackage{enumitem}
\usepackage[ruled,vlined]{algorithm2e}
\usepackage{mathrsfs}
\usepackage[gen]{eurosym}
\usepackage{graphicx}
\usepackage{comment}
\usepackage{dsfont}

\usepackage[colorinlistoftodos]{todonotes}

\theoremstyle{plain}
\newtheorem{theorem}{Theorem}[section]

\newtheorem{corollary}[theorem]{Corollary}

\newtheorem{lemma}[theorem]{Lemma}

\newtheorem{proposition}[theorem]{Proposition}

\newtheorem{definition}[theorem]{Definition}

\theoremstyle{remark}
\newtheorem{remark}[theorem]{Remark}

\newcommand{\logit}{\operatorname{logit}}

\newlist{assumptionenum}{enumerate}{1}
\setlist[assumptionenum]{
    label=(\roman*),
    ref=\theassumptionalt.(\roman*)
}
\newlist{assumptioncases}{enumerate}{1}
\setlist[assumptioncases]{
    label=(\alph*),
    ref=\theassumptionalt.(\roman{assumptionenumi}).(\alph*)
}

\numberwithin{equation}{section}

\title[Martingale diffusions as binary posteriors]{Embedding martingale diffusions as binary posteriors in sequential inference} 
\author{Steven Campbell$^a$ and Karl Kristian Engelund$^b$} 
\thanks{SC gratefully acknowledges support from an NSERC Postdoctoral Fellowship (PDF‑599675-2025) and a Columbia University CDFT Research Grant.}
\address{$^a$Dept.\ of Statistics, Columbia University, New York, NY, USA.}
\address{$^b$Dept.\ of Mathematical Sciences, University of Copenhagen, Copenhagen, Denmark}

\email{sc5314@columbia.edu}
\date{\today}

\begin{document}

\begin{abstract}
    \input{abstract}
\end{abstract}
    
	\maketitle
	\vspace{-0.5cm}

	

\input{main}
\end{document}

%% file: abstract.tex
Posterior probabilities are bounded martingales, but this fact alone does not identify the experiment that generates them. We show that every time-homogeneous martingale diffusion on $[0,1]$ with strictly positive $C^1$ volatility on the interior and absorbing endpoints, which we call an autonomous win-martingale, can be realized as the Bayesian posterior of an explicit binary sequential experiment observed through a scalar diffusion. Conversely, in a general binary diffusion experiment, the current time and observation are jointly sufficient for the hidden state precisely when the difference between the signal drifts under the two hypotheses satisfies a Riccati equation.  Equivalently, this holds when the likelihood ratio is generated by a Doob $h$-transform associated with a positive space-time harmonic function for the null dynamics. Within this class, when that function is spatially harmonic, the posterior is an autonomous martingale. Outside this harmonic case, autonomy forces the drift difference to be constant, recovering the classical two-state Shiryaev--Wonham filter from binary testing. These results give concrete sequential inference interpretations to a broad class of bounded martingale diffusions, including absorbed Brownian motion, the Wright--Fisher diffusion, and the Aldous martingale of the ``most exciting game.''

%% file: main.tex
\section{Introduction}

A $[0,1]$--valued martingale can be realized as a posterior probability in a number of ways. For instance, given such a martingale $M$ on a horizon $[0,T]$, enlarge the terminal $\sigma$--field by an independent uniform $U$ and set $\theta=\mathbf{1}_{\{U\le M_T\}}$. Invoking the tower property gives $M_t=\mathbb{P}(\theta=1\mid\mathcal{F}_t)$. A related representation is the Bernoulli–Doob approach of Brigo and Vrins \cite{brigovrin2026}, who identify the terminal Bernoulli variable with the martingale's almost-sure limit when $M_\infty\in\{0,1\}$. While constructions in this spirit settle the existence of a posterior representation, they take the martingale as given and do not address what an observer would see to form their beliefs. The question we take up is therefore a different one. Rather than ask whether a binary representation exists, we ask whether the belief can be generated by a structured observation model that an observer learns from as it evolves, and, once that structure is fixed, how much freedom the representation leaves.

The question matters because bounded martingales are routinely adopted as models in their own right. Optimality criteria based on specific distances between martingale laws, a notion going back to Gantert
\cite{gantert1991einige}, have recently produced a catalogue of canonical belief processes. Specific relative entropy selects the Aldous martingale of the ``most exciting game'' of Backhoff-Veraguas and Beiglb\"ock \cite{backhoff-beiglbock}, and related criteria single out, among others, the Bass and Wright--Fisher martingales 
\cite{backhoff2026reciprocal,backhoff2025bass,backhoff2026specific, guo2025randomness}. The idea of rewarding belief movement also appears in information design, where
Ely, Frankel and Kamenica \cite{ely2015suspense} maximize suspense, the
expected variance of the next belief. Bounded martingales also arise directly in applications. Under a pricing measure with zero interest rates, the price of a binary
contract is a $[0,1]$--valued martingale. 
Aldous \cite{aldous2013using}
develops this model for prediction markets and compares it to market data, and
such price paths have since been studied empirically \cite{augenblick2021belief}. In the same vein, conic martingales with prescribed volatility serve as models for survival and credit probabilities \cite{jeanblanc2018conic}. In each of these settings the modeler prescribes dynamics for a belief process, so a martingale may look like a posterior before anything has been said about what is observed.

Our work makes the link concrete for observers who learn through a scalar diffusion. We consider a time-homogeneous martingale diffusion with dynamics $d\Pi_t=\sigma(\Pi_t)\,dB_t$ taking values in $[0,1]$ and absorbed at the endpoints. Such a process declares a winner between the outcomes $0$ and $1$, and its coefficient depends on the state alone, so we call $\Pi$ an \emph{autonomous win-martingale}. Alongside it we place the binary sequential experiment in which a hidden state of nature $\theta\in\{0,1\}$ alters the drift of a scalar observation in Brownian noise: $dX_t=\mu_\theta(X_t)\,dt+dW_t$. Two questions organize the paper. \emph{Can a prescribed autonomous win-martingale be realized as the exact posterior in such an experiment? And which experiments of this form generate posteriors that are themselves autonomous win-martingales?}

Theorem \ref{thm:posterior-embedding} below answers the first question affirmatively and constructively. Every autonomous win-martingale of the above form is the exact Bayesian posterior of an explicit binary diffusion experiment. The construction places $\Pi$ in Lamperti coordinates, where it becomes a unit-volatility diffusion, and conditions its law on the two terminal outcomes. The resulting conditional laws are Doob $h$--transforms carrying explicit state-dependent drifts, and mixing them according to the prior assembles an experiment whose posterior is a time-independent transform of the current signal. The law of this posterior on path space is the prescribed law of $\Pi$, so the embedding recovers the full belief dynamics and not merely the marginals.
 
The same structure becomes a constraint in the reverse direction. Write $\delta=\mu_1-\mu_0$ for the drift gap. Theorem \ref{thm:recognition} shows that the posterior is a smooth function of $(t,X_t)$ if and only if $\delta$ satisfies the Riccati equation
$\delta'(x)+2\mu_0(x)\delta(x)+\delta(x)^2=-\kappa$
for some constant $\kappa$. Equivalently, the Girsanov density collapses from a functional of the entire observed path to the Doob $h$--transform generated by a positive space-time harmonic function of the null generator; that is, a function of the current time and state alone.
 
Theorem \ref{thm:classification}, and the subsequent Corollary \ref{cor:classification}, then classify autonomy. If $\kappa=0$, $X_t$ is sufficient for $\theta$ and every nontrivial experiment in this class produces an autonomous win-martingale, provided $\theta$ is terminally revealed. If $\kappa\neq0$, autonomy forces the gap to be a constant $\Delta$, and, up to the sign of the driving Brownian motion, the posterior has the classical sequential-testing volatility $\sigma(p)=|\Delta|p(1-p)$, the static two-state case of the Shiryaev--Wonham filter for a finite-state Markov chain \cite{shiryaev1967two,wonham1965some}.
 
Together, these results provide both an embedding theorem and an identification principle. They give sequential-inference interpretations to a broad family of bounded martingales, including absorbed Brownian motion, the Wright--Fisher diffusion, and the time-homogeneous version of the Aldous martingale. They also yield an inverse-design procedure that constructs an alternative drift from a chosen null model, producing autonomous posterior dynamics.

The closest related literature we are aware of concerns conic martingale diffusions \cite{brigovrin2026,jeanblanc2018conic}. Namely, in their work \cite{jeanblanc2018conic}, Jeanblanc and Vrins derive constructive representations of bounded martingales
as transformations of unconstrained latent diffusions which are useful for simulation. Moreover, in addition to the aforementioned Bernoulli-Doob representation, Brigo and Vrins \cite{brigovrin2026} study the boundary behavior and long-time limit of time-homogeneous bounded martingale diffusions and show, under a Markov and filtration assumption on the process, that every Bernoulli--Doob martingale is a diffusion. Our analysis
complements this line of work by reconstructing the statistical experiment for which the martingale is the posterior. This filtering viewpoint places the paper in the classical
tradition of sequential inference
\cite{shiryaev1967two,wald1947}. Sequential hypothesis testing is well understood for Brownian drift \cite{gapeev2004wiener,shiryaev1967two} and has been extended to more general processes and loss structures \cite{campbell2025bayesian,gapeev2011sequential}, but which posterior dynamics are attainable, and when they reduce to the tractable time-homogeneous form, has not been treated. By exhibiting a broad family
of bounded martingales as exact posteriors with explicit dynamics, our results
supply interpretable belief processes on which such problems can be posed.

\section{Win-martingales and sequential inference}\label{sec:setup}

Let $(\Omega, \mathcal{F}, \mathbb{F}=(\mathcal{F}_t)_{t\geq0}, \mathbb{P})$ be a filtered probability space satisfying the usual conditions that supports a $(\mathbb{P},\mathbb{F})$--Brownian motion $B=(B_t)_{t\geq0}$. We consider a time-homogeneous martingale diffusion $\Pi=(\Pi_t)_{t\geq0}$ taking values in $[0,1]$ and absorbed upon reaching
either endpoint. For $\sigma\in C^1((0,1);(0,\infty))$  and $\tau=\inf\{t\ge0:\Pi_t\in\{0,1\}\}$ we say that $\Pi$ is an \emph{autonomous win-martingale}  if
\begin{equation}\label{eq:intro-target}
\Pi_t=\pi+\int_0^{t\wedge\tau}\sigma(\Pi_s)dB_s,
  \qquad t\geq0, \qquad \pi\in(0,1).
\end{equation}

\begin{lemma}\label{lem:pi-well-defined}
Equation \eqref{eq:intro-target} admits a pathwise unique strong solution.  Moreover, $\Pi$ is a bounded
uniformly integrable martingale with $\Pi_t\longrightarrow\Pi_\infty\in\{0,1\}$ both almost surely and in $L^1$, and $\mathbb{P}(\Pi_\infty=1)=\pi$.
\end{lemma}

\begin{proof}
Define the Borel function $\bar{\sigma}\colon\mathbb{R}\to\mathbb{R}$ by $\bar{\sigma}(x)=
\sigma(x)$ if $x\in(0,1)$ and $\bar\sigma(x)=0$ otherwise. Since $\sigma\in C^1((0,1);(0,\infty))$, $\bar\sigma^{-2}\in L^1_{\mathrm{loc}}((0,1))$ and $\bar\sigma>0$ on $(0,1)$. Consequently, every point
at which $\bar{\sigma}^{-2}$ fails to be locally integrable is a zero
of $\bar{\sigma}$.
The Engelbert--Schmidt theorem
\cite[Theorem 5.5.4]{karatzas-shreve} yields a weak
non-explosive solution of $dZ_t=\bar{\sigma}(Z_t)\,dB_t$,  $Z_0=\pi$.
Stopping at its first exit time from $(0,1)$ and absorbing it
there gives a weak absorbed solution of \eqref{eq:intro-target}. Since $\sigma\in C^1((0,1))$, it is locally Lipschitz. A localization
of \cite[Theorem 5.2.5]{karatzas-shreve}, together with
\cite[Remark 5.3.3]{karatzas-shreve}, therefore gives pathwise
uniqueness up to the exit time from $(0,1)$. After absorption at the exit point, pathwise uniqueness holds. A localization of the Yamada--Watanabe theorem
\cite[Corollary 5.3.23]{karatzas-shreve} then yields a pathwise unique
strong solution to \eqref{eq:intro-target}.

Prior to absorption, $\Pi$
is a diffusion on $(0,1)$ in natural scale, with scale function $p(x)=x$.
The hypotheses of
\cite[Proposition 5.5.22(d)]{karatzas-shreve} are satisfied since
$\sigma>0$ and $\sigma^{-2}\in L^1_{\mathrm{loc}}((0,1))$. Since $p(0+)=0$ and $p(1-)=1$,
that proposition, applied to the diffusion on $(0,1)$ up to its
exit $\tau$, yields $\mathbb{P}\left(\lim_{t\uparrow\tau}\Pi_t=0\right)
  =
  1-\pi$, and $\mathbb{P}\left(\lim_{t\uparrow\tau}\Pi_t=1\right)
  =
  \pi$.
Here, $t\uparrow\tau$ is understood as $t\to\infty$ on
$\{\tau=\infty\}$. Since the process is absorbed when it reaches either
endpoint, it follows that $\Pi_t\rightarrow\Pi_\infty\in\{0,1\}$ a.s.~and $\mathbb{P}(\Pi_\infty=1)=\pi$.
As $0\leq\Pi_t\leq1$, dominated convergence implies that
the convergence also holds in $L^1$. 
\end{proof}

\subsection{Binary sequential inference}\label{sec:bin-seq-inf}
To connect $\Pi$ to a statistical problem, we introduce 
the classical setup in continuous-time Bayesian sequential inference. Consider a (possibly different) probability space $(\widetilde{\Omega}, \widetilde{\mathcal{F}},\widetilde{\mathbb{F}}=(\widetilde{\mathcal{F}}_t)_{t\geq0}, \widetilde{\mathbb{P}})$ with $\widetilde{\mathbb{P}}$--expectation $\widetilde{\mathbb{E}}[\cdot]$. Fix an open interval $J=(\ell,r)\subseteq\mathbb R$ with $-\infty\leq\ell<r\leq\infty$ and $x_0\in J$.  We write $\overline{\mathbb R}:=[-\infty,\infty]$ for the extended real line
and denote by $\overline J:=[\ell,r]\subseteq\overline{\mathbb R}$ the extended closure of $J$, equipped with the order topology.

Suppose $\theta\in\{0,1\}$ is a binary $\widetilde{\mathcal{F}}_0$--measurable hidden state with $\theta\sim\mathrm{Bernoulli}(p)$, $p\in(0,1)$. An observer has access to an $\widetilde{\mathbb{F}}$--adapted signal process $X=(X_t)_{t\geq0}$  with continuous
paths (in the order topology) taking values in $\overline{J}$ and absorbed upon reaching $\ell$ or $r$.
Letting $\zeta=\inf\{t\geq0:X_t\not\in J\}$ we assume the evolution
\begin{equation}\label{eq:signal-process}
  X_t=\begin{cases}
      x_0 + \int_0^{t}\mu_\theta(X_s)ds+W_{t}, & t<\zeta\\
      X_\zeta, & t\geq \zeta
  \end{cases}, \qquad t\geq0,
\end{equation}
where $W$ is a $(\widetilde{\mathbb{P}},\widetilde{\mathbb{F}})$--Brownian motion independent of $\theta$ and $\mu_0(\cdot),\mu_1(\cdot)\in C(J;\mathbb{R})$ govern the dynamics of the signal in each state of nature. On $\{\zeta<\infty\}$, $X_\zeta\in\{\ell,r\}$ denotes the endpoint reached in the extended real line.%

\begin{lemma}\label{lem:signal-well-posed}
Equation
\eqref{eq:signal-process} admits a pathwise unique strong solution.
\end{lemma}

\begin{proof}
Fix $i\in\{0,1\}$ and define the scale function $s_i(x)
  :=
  \int_{x_0}^x
  \exp\left(
    -2\int_{x_0}^y\mu_i(z)\,dz
  \right)dy$ for $x\in J$.
Then $s_i$ is a strictly increasing $C^2$--diffeomorphism from $J$
onto the open interval $I_i:=s_i(J)$. Set $\sigma_i:=s_i'\circ s_i^{-1}$.
Since $s_i''(x)=-2\mu_i(x)s_i'(x)$,  differentiating gives $\sigma_i'(y)
  =
  s_i''(s_i^{-1}(y))/
       s_i'(s_i^{-1}(y))
  =
  -2\mu_i(s_i^{-1}(y))$ for $y\in I_i$,
and so $\sigma_i\in C^1(I_i;(0,\infty))$. Define $\bar{\sigma}_i(x) = \sigma_i(x)$ for $x\in I_i$ and $\bar{\sigma}_i(x)=0$ otherwise. Arguing as in Lemma \ref{lem:pi-well-defined} we obtain a unique strong solution to the auxiliary equation $dY_t^i=\bar\sigma_i(Y_t^i)\,dW_t$, $Y_0^i=s_i(x_0)$ with absorption at $\zeta_i:=\inf\{t\geq0:Y_t^i\notin I_i\}$.

Note that $s_i$ extends uniquely to the extended closures of $J$ and $I_i$. It\^o's formula then shows that $X_t^i:=s_i^{-1}(Y_{t\wedge \zeta_i}^i)$ satisfies \eqref{eq:signal-process} with $\theta=i$.
Moreover, through its relation with $Y^i$, it is the pathwise unique strong solution of this equation with absorption.
Set $\zeta:=\zeta_0$, $X_t:=X_t^0$ on $\{\theta=0\}$ and $\zeta:=\zeta_1$, $X_t:=X_t^1$ on $\{\theta=1\}$
with $X_t = X_\zeta$ for $t\geq\zeta$.  Since
$\{\theta=i\}\in\widetilde{\mathcal F}_0$, this construction of $X$ is an $\widetilde{\mathbb F}$--adapted strong solution of \eqref{eq:signal-process}. Moreover, pathwise uniqueness for
each $X^i$ implies pathwise uniqueness for $X$.
\end{proof}

Let $\mathbb F^X=(\mathcal F_t^X)_{t\geq0}$ be the usual augmentation,
under $\widetilde{\mathbb P}$, of the natural filtration of $X$. In the context of sequential inference, we are interested in our posterior belief,  $P_t=\widetilde{\mathbb{P}}(\theta=1\mid\mathcal{F}_t^X)=  \widetilde{\mathbb E}[\theta\mid\mathcal F_t^X]$, $t\geq0$,
for which $P_0=p$. Throughout, we identify $P$ with the c\`adl\`ag version of this $[0,1]$--valued martingale. Our next proposition will provide a characterization of this posterior process.

We will need some notation. For
$i\in\{0,1\}$, let
$\widetilde{\mathbb{Q}}^i:=\widetilde{\mathbb P}(\,\cdot\,\vert\theta=i)$ be the conditional measure given $\theta=i$ and set
$\delta:=\mu_1-\mu_0$. Given a compact subset $K=[a,b]$ of $J$ with $x_0\in(a,b)$, we write
$\tau_K:=\inf\{t\geq0:X_t\notin(a,b)\}$. A positive
$\mathbb F^X$--adapted process $Z$ will be called a local
$\widetilde{\mathbb{Q}}^1/\widetilde{\mathbb{Q}}^0$--likelihood-ratio process before $\zeta$ if
$Z_{t\wedge\tau_K}$ is the $\widetilde{\mathbb{Q}}^0$--density of $\widetilde{\mathbb{Q}}^1$ on
$\mathcal F^X_{t\wedge\tau_K}$ for every such $K$ and every finite $t$.

For a stopping time $\tau$, we use $\llbracket 0,\tau\llbracket
  :=
  \{(t,\omega)\in[0,\infty)\times\widetilde\Omega:
    0\leq t<\tau(\omega)\}$
to denote the stochastic interval before $\tau$. Statements holding
indistinguishably on $\llbracket0,\tau\llbracket$ are understood to
hold, outside a single $\widetilde{\mathbb P}$--null set, for every
$t<\tau(\omega)$.
On $\llbracket0,\zeta\llbracket$, define
\begin{equation}\label{eq:likelihood-process}
  L_t
  :=
  \exp\left\{
    \int_0^t\delta(X_s)\,dX_s
    -
    \frac12\int_0^t
      \bigl(\mu_1^2-\mu_0^2\bigr)(X_s)\,ds
  \right\},
  \qquad t<\zeta,
\end{equation}
\begin{equation}\label{eq:mixed-dynamics-general}
  \widehat\mu(X_t, P_t)
  :=
  \mu_0(X_t)+P_t\delta(X_t),
  \quad \text{and} \quad
  \overline W_t
  :=
  X_t-x_0-\int_0^t\widehat\mu(X_s, P_s)\,ds,
  \quad t<\zeta.
\end{equation}
These quantities are well defined since the path of $X$ is contained in $J$ whenever $t<\zeta$.

\begin{proposition}
\label{prop:likelihood-posterior}
The process $L$ is the local $\widetilde{\mathbb{Q}}^1/\widetilde{\mathbb{Q}}^0$--likelihood-ratio process
before $\zeta$ and, indistinguishably on
$\llbracket0,\zeta\llbracket$,
\begin{equation}\label{eq:posterior-odds}
  P_t
  =
  \frac{pL_t}{1-p+pL_t},
  \qquad
  \frac{P_t}{1-P_t}
  =
  \frac{p}{1-p}L_t.
\end{equation}
Moreover, $\overline W$ is an
$(\widetilde{\mathbb P},\mathbb F^X)$--Brownian motion on
$\llbracket0,\zeta\llbracket$ called the \emph{innovation process} and
\begin{equation}\label{eq:filtering-equation}
  P_t
  =
  p+\int_0^t
    P_s(1-P_s)\delta(X_s)\,d\overline W_s,
  \qquad t<\zeta.
\end{equation}
\end{proposition}

\begin{proof}
This is a localization of standard results in filtering theory (cf.~\cite{bain-crisan,liptser-shiryaev}) %
that can be justified as follows. Fix a compact $K=[a,b]$ in $J$ as above with $x_0\in(a,b)$ and define bounded
continuous extensions of the drifts by
$\mu_i^{(K)}(x):=\mu_i((x\vee a)\wedge b)$, $x\in\mathbb R$ and
$i\in\{0,1\}$. Let $X^{(K)}$ be the global signal obtained by
replacing $\mu_i$ with $\mu_i^{(K)}$ and $J$ with $\mathbb R$, driven
by the same pair $(\theta,W)$, and define
$P^{(K)},L^{(K)}$ and $\overline W^{(K)}$ analogously.

On every \emph{finite horizon}, the likelihood formula
\cite[Theorem~7.1 and the subsequent corollary]{liptser-shiryaev},
applied under $\widetilde{\mathbb Q}^0$ and $\widetilde{\mathbb Q}^1$ and
followed by Bayes' formula, gives the likelihood-ratio and posterior
identities in \eqref{eq:posterior-odds} for the bounded model.
Moreover, the general filtering theorem
\cite[Theorem~8.1]{liptser-shiryaev}, applied with estimated process
$h_t\equiv\theta$, observation drift
$A_t=\mu_\theta^{(K)}(X_t^{(K)})$, and $B\equiv1$, shows that
$\overline W^{(K)}$ is an observation Brownian motion and that
\[
  dP_t^{(K)}
  =
  P_t^{(K)}(1-P_t^{(K)})
  \bigl(\mu_1^{(K)}-\mu_0^{(K)}\bigr)(X_t^{(K)})
  \,d\overline W_t^{(K)}.
\]
In the notation of that theorem, $H=D=0$ so the expression reduces to the above.

Let
$\tau^{(K)}:=\inf\{t\geq0:X_t^{(K)}\notin(a,b)\}$. Since the original
and extended coefficients agree on $K$, pathwise uniqueness gives
$\tau^{(K)}=\tau_K$ and
$X^{(K)}_{t\wedge\tau_K}=X_{t\wedge\tau_K}$ for every $t\geq0$.
The stopped-filtration identity
\cite[Lemma~2.32]{bain-crisan} therefore gives
$\mathcal F_{t\wedge\tau_K}^{X^{(K)}}
 =\mathcal F_{t\wedge\tau_K}^{X}$, and optional sampling yields
$P_{t\wedge\tau_K}^{(K)}=P_{t\wedge\tau_K}$. The definitions then
also give
$L_{t\wedge\tau_K}^{(K)}=L_{t\wedge\tau_K}$ and
$\overline W_{t\wedge\tau_K}^{(K)}
 =\overline W_{t\wedge\tau_K}$. Hence all the asserted identities
hold up to $\tau_K$.

Finally, along an increasing compact exhaustion of $J$, the
corresponding exit times increase to $\zeta$, and the stopped
identities agree on overlaps; ergo on
$\llbracket0,\zeta\llbracket$.
\end{proof}

\begin{remark}\label{rmk:consistency}
Since $P$ is a bounded martingale, there exists an
$\mathcal F_\infty^X$--measurable random variable $P_\infty$ such that
$P_t\to P_\infty$ a.s.~and in $L^1$, where
$\mathcal F_\infty^X=\bigvee_{t\geq0} \mathcal{F}^X_t$.
L\'evy's upward theorem gives $P_\infty=\widetilde{\mathbb E}
  \bigl[\theta\mid\mathcal F_\infty^X\bigr]$.
Consequently, if $P_\infty\in\{0,1\}$ a.s.,~then $\widetilde{\mathbb E}
  \left[
    (\theta-P_\infty)^2
    \,\middle|\,
    \mathcal F_\infty^X
  \right]
  =
  P_\infty(1-P_\infty)
  =
  0$,
and hence $P_\infty=\theta$ a.s. A sufficient condition for $P_\infty\in\{0,1\}$ is that $\int_0^\zeta \delta(X_s)^2\,ds=\infty$ a.s.~(understood as an improper integral on $\llbracket 0, \zeta\llbracket$). Indeed, under $\widetilde{\mathbb Q}^i$, $i\in\{0,1\}$,
$\log L_t
  =
  M_t^i
  +
  \left(i-\frac12\right)
  \int_0^t\delta(X_s)^2ds$ for $t<\zeta$
where $M^i$ is a continuous local martingale with $\langle M^i\rangle_t
  =
  \int_0^t\delta(X_s)^2ds$.
The Dambis--Dubins--Schwarz theorem
\cite[Theorem~3.4.6]{karatzas-shreve} and the Brownian strong law
therefore imply that $L_t\to0$ under $\widetilde{\mathbb Q}^0$ and
$L_t\to\infty$ under $\widetilde{\mathbb Q}^1$ as $t\uparrow\zeta$.
The identity \eqref{eq:posterior-odds} then yields
$P_\infty=\theta$ a.s.
\end{remark}

The goal of the ensuing sections is to relate autonomous win-martingales to posterior processes arising from binary sequential inference. The preceding lemmas establish the well-posedness of both models. In particular, they provide pathwise unique strong solutions to \eqref{eq:intro-target} and \eqref{eq:signal-process}, and hence uniquely determined laws on the corresponding path spaces. This allows us to treat each model as a canonical object.

\begin{remark}
    It is possible to relax the regularity requirements on $\sigma$, $\mu_0$, and $\mu_1$. Indeed, most of the arguments admit extensions when $\sigma$ is positive and locally absolutely continuous, $\mu_0,\mu_1$ are finite
    Borel functions in $L^1_{\mathrm{loc}}(J)$, and
    $\delta=\mu_1-\mu_0$ is locally bounded. 
    We have abstained from adopting this as the default setting since the main ideas are essentially unchanged, but the technical burden is substantially increased.
\end{remark}

\subsection{Doob \texorpdfstring{$h$}{h}--transforms}\label{sec:doob}
Before proceeding we briefly recall the form of the Doob $h$--transform that appears repeatedly in what follows; see, for example, \cite{rogerswilliams2000}. Let $(\Omega^\circ,\mathcal A^\circ,\mathbb F^\circ,\mathbb P^\circ)$ be a filtered probability space satisfying the usual conditions, where $\mathbb F^\circ=(\mathcal A_t^\circ)_{t\geq0}$, and let $B$ be an $(\mathbb F^\circ, \mathbb P^\circ)$--Brownian motion. Suppose that, for some open interval $I\subseteq\mathbb R$ and $y\in I$, the diffusion $Y$ satisfies $dY_t=\nu(Y_t)\,dt+dB_t$, $Y_0=y$, up to its lifetime $\eta:=\inf\{t\geq0:Y_t\notin I\}$ where $\nu\in C(I;\mathbb{R})$. For every compact interval $K=[a,b]\subset I$ with $y\in(a,b)$, let $\eta_K:=\inf\{t\geq0:Y_t\notin(a,b)\}$.
As in the preceding setup, all changes of measure and stochastic dynamics below are understood locally before $\eta$, by stopping at $\eta_K$. 

A function \(\widehat h\in C^{1,2}([0,\infty)\times I;(0,\infty))\) is called \emph{space-time} harmonic for the generator $\mathcal L^\circ f = \nu f'+\frac12f''$ of $Y$ on $I$ if it satisfies $(\partial_t+\mathcal L^\circ)\widehat h=0$. In this case, It\^o's formula shows that $Z_t
:=
\widehat h(t,Y_t)/\widehat h(0,y)$
for $t<\eta$ is a positive local martingale with $Z_0=1$. For every compact \(K\) as above $Z_{t\wedge\eta_K}$ is a true exponential martingale on every finite horizon and therefore defines a change of measure on $\mathcal A_{t\wedge\eta_K}^\circ$. The resulting change of measure before $\eta$ is called the \emph{Doob $\widehat h$--transform} of $\mathbb P^\circ$.
If $Z$ is a true martingale, then $(\mathrm d\mathbb P^{\circ,\widehat h}/\mathrm d\mathbb P^\circ)|_{\mathcal A_t^\circ}=Z_t$, $t\geq0$, defines a consistent family of probability measures on $(\mathcal A_t^\circ)_{t\geq0}$.
A special case arises when $h\in C^2(I;(0,\infty))$ satisfies $\mathcal L^\circ h=\lambda h$
for some $\lambda\in\mathbb R$. Then $\widehat h(t,x)=e^{-\lambda t}h(x)$
is space-time harmonic, and the corresponding density is
$Z_t
=
e^{-\lambda t}h(Y_t)/h(y)
$.
When $\lambda=0$, the function $h$ is $\mathcal L^\circ$--harmonic and the exponential factor disappears.

The space-time transform has a conditioning interpretation whenever its density is generated by a positive-probability event. In the globally defined case, suppose $E\in\mathcal A^\circ$ satisfies $\mathbb P^\circ(E)>0$ and $\mathbb P^\circ(E\mid\mathcal A_t^\circ)=\widehat h(t,Y_t)$ for every $t\geq0$. Then $\mathbb P^\circ(E)=\widehat h(0,y)$ and
\begin{equation}\label{eq:doob-conditioning}
    \mathbb P^\circ(A\mid E)
    =
    \mathbb E^\circ\left[
        \mathbf 1_A
        (\widehat h(t,Y_t) /\widehat h(0,y))
    \right], \qquad \forall A\in\mathcal{A}^\circ_t.
\end{equation}
Thus, on each $\mathcal A_t^\circ$, conditioning on $E$ coincides with the Doob $\widehat h$--transform.

\section{Statistical embedding of win-martingales}\label{sec:embedding}

Our posterior embedding is constructive and provides a satisfying intuition for the relationship to sequential inference. Our starting point is $\Pi$ from \eqref{eq:intro-target} and we inherit all the notation of Section \ref{sec:setup} here. 

The first step is to transform $\Pi$ so that the resulting process is a Markovian diffusion with unit volatility that resembles $X$. In order to do this, we will need to introduce two auxiliary functions. Let $F$ be a Lamperti transform associated with $\Pi$, defined by 
\begin{equation}\label{eq:F-def}
  F(p)=\int_{1/2}^p \frac{1}{\sigma(u)}du+C, \qquad p\in(0,1),
\end{equation}
for some $C\in\mathbb R$.
Its image $I=F((0,1))$ is an open interval with possibly infinite endpoints $-\infty\leq \lambda< \rho\leq\infty$, and $F:(0,1)\to I$ is a $C^2$ diffeomorphism. Denote its inverse by $G:I\to(0,1)$ and note that $G$ can naturally be extended to $\overline{I}=[\lambda,\rho]\subseteq \overline{\mathbb{R}}$ with the identification $G(\lambda) = 0$ and $G(\rho)=1$. $F$ extends analogously to $[0,1]$.

Define the candidate process $Y_t:=F(\Pi_t)$. Since $F$ is invertible, the augmented natural filtrations of $Y$ and
$\Pi$ coincide. Moreover, $F'(p)=1/\sigma(p)$ and $F''(p)=-\sigma'(p)/\sigma^2(p)$,
so It\^o's formula gives
\begin{equation}\label{eq:Y-under-P}
  dY_t
  =
  -\frac12\sigma'(G(Y_t))\,dt+dB_t,
  \qquad t<\tau.
\end{equation}
Thus, prior to absorption, $Y$ is a time-homogeneous Markov diffusion
with unit diffusion coefficient.
In order for $Y$ to play the role of the signal $X$ in Section \ref{sec:bin-seq-inf}, we need to construct two new measures $\mathbb{Q}^0$ and $\mathbb{Q}^1$ on $(\Omega, \mathcal{F},\mathbb{F})$ that will serve as $\widetilde{\mathbb{Q}}^0$ and $\widetilde{\mathbb{Q}}^1$. 

To this end we draw inspiration from the consistency property of the posterior $P$ (Remark \ref{rmk:consistency}) when there is enough information in the experiment. We observe that,
in an experiment that ultimately reveals the truth, conditioning on the hidden state is equivalent to conditioning on the terminal value reached by the posterior process. 

This suggests reversing the construction. Let $\mathbb{F}^\Pi=(\mathcal{F}^\Pi_t)_{t\geq0}$ be the augmented natural filtration of $\Pi$. Starting from $\Pi$ in \eqref{eq:intro-target}, we use its terminal value to define the two conditional laws that will subsequently be reassembled as the two states of a sequential experiment. By
Lemma \ref{lem:pi-well-defined},
$\Pi_\infty\in\{0,1\}$ a.s.~and
$\mathbb P(\Pi_\infty=1)=\pi$. Moreover, uniform integrability gives $\Pi_t
  =
  \mathbb E
  \bigl[\Pi_\infty\mid\mathcal F_t^\Pi\bigr]
  =
  \mathbb P
  \bigl(\Pi_\infty=1\mid\mathcal F_t^\Pi\bigr)$.
Thus, $\Pi$ is already the posterior process associated with its own
terminal value in the filtration $\mathbb F^\Pi$. 

In view of this, the natural choices for $\mathbb Q^0$ and
$\mathbb Q^1$ are the conditional laws under which $\Pi_{\infty}$ is $0$ and $1$, respectively. For $A\in\mathcal F$,
define $\mathbb Q^1(A)
  :=
  \mathbb P(A\mid\Pi_\infty=1)$, and $\mathbb Q^0(A)
  :=
  \mathbb P(A\mid\Pi_\infty=0)$.
Since $\Pi_\infty\in\{0,1\}$ a.s.~and
$\mathbb P(\Pi_\infty=1)=\pi$, for every $A\in\mathcal F$,
\[
\begin{aligned}
  \mathbb Q^1(A)
  &=
  \frac{\mathbb P(A\cap\{\Pi_\infty=1\})}{\pi}
  =
  \frac{1}{\pi}
  \mathbb E\left[
    \mathbf 1_A\mathbf 1_{\{\Pi_\infty=1\}}
  \right]=
  \mathbb E\left[
    \mathbf 1_A\frac{\Pi_\infty}{\pi}
  \right],
\end{aligned}
\]
and similarly, $\mathbb Q^0(A)
  =
  \mathbb E\left[
    \mathbf 1_A (1-\pi)^{-1}(1-\Pi_\infty)
  \right]$.
Consequently, 
\begin{equation}\label{eq:embedding-terminal-densities}
  \frac{d\mathbb Q^1}{d\mathbb P}
  =
  \frac{\Pi_\infty}{\pi},
  \qquad
  \frac{d\mathbb Q^0}{d\mathbb P}
  =
  \frac{1-\Pi_\infty}{1-\pi}, \quad \text{and} \quad \mathbb P=(1-\pi)\mathbb Q^0+\pi\mathbb Q^1.
\end{equation}
Taking conditional expectations in
\eqref{eq:embedding-terminal-densities} and using uniform integrability yields
\begin{equation}\label{eq:embedding-densities}
  Z_t^1
  :=
  \frac{d\mathbb Q^1}{d\mathbb P}
  \bigg|_{\mathcal F_t}
  =
  \frac{\Pi_t}{\pi},
  \qquad
  Z_t^0
  :=
  \frac{d\mathbb Q^0}{d\mathbb P}
  \bigg|_{\mathcal F_t}
  =
  \frac{1-\Pi_t}{1-\pi}.
\end{equation}

\begin{remark}\label{rmk:doob_h_embedding}
The restrictions of the preceding conditional laws to $\mathbb F^\Pi$ are the classical \emph{Doob $h$--transforms} of $\Pi$ (see Section \ref{sec:doob}). Indeed, $h_1(x):=x$ and $h_0(x):=1-x$ are strictly positive and $\mathcal L$--harmonic on
$(0,1)$, where $\mathcal Lf(x)=\frac12\sigma^2(x)f''(x)$ is the generator of $\Pi$ in $(0,1)$. As $\mathbb P(\Pi_\infty=i\mid\mathcal F_t^\Pi) =
  h_i(\Pi_t)$ for $i\in\{0,1\}$, \eqref{eq:doob-conditioning} shows that the Doob $h_i$--transform coincides with conditioning the absorbed diffusion on the event $\{\Pi_\infty=i\}$.
\end{remark}

We now identify the dynamics of $Y$ under $\mathbb{Q}^0$ and $\mathbb{Q}^1$. For $n$ sufficiently large that
$\pi\in(1/n,1-1/n)$, let $\tau_n
  :=
  \inf\{t\geq0:\Pi_t\notin(1/n,1-1/n)\}$.
Then $\tau_n\uparrow\tau$. By optional sampling,
$Z^i_{t\wedge\tau_n}$ is the density of $\mathbb Q^i$ with respect to
$\mathbb P$ on $\mathcal F_{t\wedge\tau_n}$. Moreover,
\[
  Z^1_{\cdot\wedge\tau_n}
  =
  \mathcal E\left(
    \int_0^{\cdot\wedge\tau_n}
    \frac{\sigma(\Pi_s)}{\Pi_s}\,dB_s
  \right), \quad \text{and} \quad
  Z^0_{\cdot\wedge\tau_n}
  =
  \mathcal E\left(
    -\int_0^{\cdot\wedge\tau_n}
    \frac{\sigma(\Pi_s)}{1-\Pi_s}\,dB_s
  \right).
\]
These stopped density processes are bounded martingales. Define
\[
  B^1_{t\wedge\tau_n}=B_{t\wedge\tau_n}-\int_0^{t\wedge\tau_n}\frac{\sigma(\Pi_s)}{\Pi_s}ds \quad \text{and} \quad B^0_{t\wedge\tau_n}=B_{t\wedge\tau_n}+\int_0^{t\wedge\tau_n}\frac{\sigma(\Pi_s)}{1-\Pi_s}ds.
\]
Girsanov's theorem therefore shows that $B^1$ and $B^0$ are $\mathbb{Q}^1$-- and $\mathbb{Q}^0$--Brownian motions stopped at $\tau_n$, respectively. Substituting $\Pi_t=G(Y_t)$ into \eqref{eq:Y-under-P} gives $dY_t=\nu_1(Y_t)dt+dB^1_t$ under $\mathbb{Q}^1$, and $dY_t=\nu_0(Y_t)dt+dB^0_t$ under $\mathbb{Q}^0$,
up to $\tau_n$ where, for $y\in I$,
\begin{equation}\label{eq:embedding-drifts}
  \nu_1(y)=-\frac12\sigma'(G(y))+\frac{\sigma(G(y))}{G(y)},
  \qquad
  \nu_0(y)=-\frac12\sigma'(G(y))-\frac{\sigma(G(y))}{1-G(y)}.
\end{equation}
Since $\tau_n\uparrow\tau$, the stopped Brownian motions agree on
overlapping stochastic intervals. Hence the
preceding dynamics hold locally on $[0,\tau)$ under
$\mathbb Q^1$ and $\mathbb Q^0$, respectively. In particular,
$\nu_0,\nu_1\in C(I;\mathbb R)$, so the two conditional laws have the
form required of the signal model in
Section \ref{sec:bin-seq-inf}.

This analysis suggests taking the transformed drifts $\nu_0$ and $\nu_1$ as the two state-dependent signal drifts. The next theorem shows that this choice exactly recovers the prescribed autonomous win-martingale as the posterior process.

\begin{theorem}[Posterior embedding]
\label{thm:posterior-embedding}
Let $\Pi$ be the autonomous win-martingale in
\eqref{eq:intro-target}, and let $F$, $G$, $I$, $\nu_0$, and $\nu_1$
be defined by \eqref{eq:F-def} and \eqref{eq:embedding-drifts}.
In the binary sequential experiment of
Section \ref{sec:bin-seq-inf}, take $p=\pi$, $J=I$, $x_0=F(\pi)$, and $\mu_i\equiv\nu_i$, $i\in\{0,1\}$.
Then, $P=G(X)$ up to indistinguishability,
$\zeta=\inf\{t\geq0:P_t\in\{0,1\}\}$,
\begin{equation}\label{eq:thm-P-embedding}
  P_t=p+\int_0^{t\wedge \zeta}\sigma(P_s)\,d\overline W_s, \quad t\geq0, \quad \text{and} \quad \operatorname{Law}_{\widetilde{\mathbb P}}(P)=\operatorname{Law}_{\mathbb P}(\Pi).
\end{equation}
\end{theorem}

\begin{proof}
By Lemma \ref{lem:signal-well-posed}, the prescribed experiment is
well defined. Recall that $\widetilde{\mathbb Q}^i:=\widetilde{\mathbb P}(\,\cdot\mid\theta=i)$ for $i\in\{0,1\}$. Under $\widetilde{\mathbb Q}^i$, the signal $X$ solves the absorbed equation $dX_t=\nu_i(X_t)\,dt+dW_t$ for $t<\zeta$.
At the same time, the preceding Girsanov calculation shows that $Y=F(\Pi)$ satisfies
the same equation under $\mathbb Q^i$. Uniqueness in law therefore gives $\operatorname{Law}_{\widetilde{\mathbb Q}^i}(X)=\operatorname{Law}_{\mathbb Q^i}(Y)$, $i\in\{0,1\}$. Since $p=\pi$ and $\widetilde{\mathbb P}
  =
  (1-\pi)\widetilde{\mathbb Q}^0
  +
  \pi\widetilde{\mathbb Q}^1$, $\mathbb P
  =
  (1-\pi)\mathbb Q^0+\pi\mathbb Q^1$,
it follows that $\operatorname{Law}_{\widetilde{\mathbb P}}(X)=\operatorname{Law}_{\mathbb P}(Y)$.

Fix $t\geq0$ and let $\Phi$ be any bounded Borel functional of the signal path up to time $t$. Since $\Pi=G(Y)$, we have by the density
identities \eqref{eq:embedding-densities}
\[
  \widetilde{\mathbb E}\bigl[\theta\Phi(X)\bigr]
  =
  \pi\,
  \mathbb E_{\widetilde{\mathbb Q}^1}\bigl[\Phi(X)\bigr] =
  \pi\,
  \mathbb E_{\mathbb Q^1}\bigl[\Phi(Y)\bigr] =
  \mathbb E\bigl[\Pi_t\Phi(Y)\bigr] =
  \widetilde{\mathbb E}
  \bigl[G(X_t)\Phi(X)\bigr].
\]
It follows that $P_t= \widetilde{\mathbb E}[\theta\mid\mathcal F_t^X]=G(X_t)$ $\widetilde{\mathbb P}$--a.s.
Equality at rational times, together with the c\`adl\`ag property of
$P$ and continuity of $G(X)$, gives indistinguishability. Since $G(I)=(0,1)$ and $G$ maps the two endpoints of $\overline I$
to $0$ and $1$, respectively, $\zeta
  =
  \inf\{t\geq0:P_t\in\{0,1\}\}$,
and $P=G(X)$ is absorbed after $\zeta$. Finally, the definitions of $\nu_0$ and $\nu_1$ give $\delta(x):=
  \nu_1(x)-\nu_0(x) 
  =
  \sigma(G(x))/(G(x)(1-G(x)))$.
Using $P_t=G(X_t)$ in the filtering equation
\eqref{eq:filtering-equation}, we obtain $dP_t
  =
  P_t(1-P_t)\delta(X_t)\,d\overline W_t=
  \sigma(P_t)\,d\overline W_t$, $t<\zeta$,
and \eqref{eq:thm-P-embedding} follows by absorption. The equality of laws is immediate.
\end{proof}

\section{Win-martingales from binary experiments}\label{sec:converse}
We now turn to the converse, in which a posterior is induced by the observation process of a binary sequential experiment. We adopt the notation of Section \ref{sec:bin-seq-inf} and assume that $\delta=\mu_1-\mu_0\in C^1(J;\mathbb R)$. This assumption is consistent with the preceding section. Although the individual drifts produced by the embedding need only be continuous, their difference is $C^1$. Throughout this section, stochastic identities are understood to hold up to indistinguishability on $\llbracket 0,\zeta\llbracket$ unless otherwise stated, and arguments are localized at $\tau_K$, for compact interval $K\subset J$ with $x_0$ in the interior of $K$, whenever necessary.

By \eqref{eq:likelihood-process} and \eqref{eq:posterior-odds}, the posterior $P_t$ is determined by $t$ and the observed path $(X_s)_{s\leq t}$. This is weaker than the relation considered in the preceding section, where the observation process $Y$ was obtained from $\Pi$ through an \emph{invertible} pointwise transformation: $Y_t=F(\Pi_t)$, $\Pi_t = G(Y_t)$. To obtain a genuine converse to that construction, we first seek conditions under which the posterior can be recovered solely from the current observation, rather than its entire history. This leads to the following definition, which ensures that the pair $(t,X_t)$ is a sufficient statistic for the hidden state $\theta$.

\begin{definition}[Current-state posteriors]\label{def:current-state}
The posterior is a \emph{current-state (C-S) posterior} if $P_t=\Gamma(t,X_t)$ on $\llbracket 0,\zeta\llbracket$ for some
\(
  \Gamma\in C([0,\infty)\times J;(0,1))\cap C^{1,2}((0,\infty)\times J;(0,1)).
\)
\end{definition}

We now characterize these posteriors fully. We will need the following support property to pass from stochastic identities evaluated along $X$ to pointwise identities.

\begin{lemma}
\label{lem:stopped-full-support}
    For all $s>0$ and every nonempty open $U\subset\operatorname{int}(K)$, $\widetilde{\mathbb P}(X_s\in U,\ s<\tau_K)>0$.
\end{lemma}

\begin{proof}
    Under $\widetilde{\mathbb P}$, $dX_t=\widehat{\mu}(X_t,P_t)\,dt+d\overline W_t$, $t<\zeta$.
    The stopped drift $\beta_t:=\mathbf 1_{\{t<\tau_K\}}\widehat{\mu}(X_t,P_t)$ is bounded, since $P_t\in[0,1]$ and $\mu_0$ and $\delta$ are bounded on $K$. Hence, for every fixed $s>0$, Girsanov's theorem, with Novikov's condition on $[0,s\wedge\tau_K]$, yields a probability measure $\widehat{\mathbb P}$ equivalent to $\widetilde{\mathbb P}$ on $\mathcal F^X_s$ such that $X_{\cdot\wedge\tau_K}$ is a stopped $(\widehat{\mathbb{P}},\mathbb F^X)$--Brownian motion started at $x_0$. The full-support property of Brownian motion gives $ \widehat{\mathbb P}(X_s\in U,\ s<\tau_K)>0$.
    Since $\widehat{\mathbb P}$ and $\widetilde{\mathbb P}$ are equivalent on $\mathcal F^X_s$, the claim follows.
\end{proof}

\begin{theorem}[C-S posterior identification]\label{thm:recognition}
$P$ is a C-S posterior if and only if $\delta$ satisfies 
\begin{equation}\label{eq:riccati}
  \delta'(x)+2\mu_0(x)\delta(x)+\delta(x)^2=-\kappa,
  \qquad x\in J,
\end{equation}
for some $\kappa\in\mathbb{R}$.
In this case, the local $\widetilde{\mathbb{Q}}^1/\widetilde{\mathbb{Q}}^0$--likelihood-ratio process before $\zeta$ is
\begin{equation}\label{eq:path-independent-L}
  L_t=\exp\left\{
    A(X_{t})+\kappa t/2
  \right\},
  \qquad t<\zeta, \qquad 
  A(x):=\int_{x_0}^x\delta(u)\,du, \quad \text{and}
\end{equation}
\begin{equation}\label{eq:current-state-map}
 P_t=\Gamma(t,X_t), \ \ \text{where} \ \ \Gamma(t,x)=H\left(\logit p+A(x)+\kappa t/2\right),
  \ \
  H(r):=1/(1+e^{-r}).
\end{equation}
\end{theorem}
\begin{proof}
Suppose first that $P$ is a C-S posterior, so that
$P_t=\Gamma(t,X_t)$, and put $R=\logit \Gamma$. Comparing It\^o's formula for
$\Gamma(t,X_t)$ with \eqref{eq:filtering-equation} gives
\begin{equation}\label{eq:G-Ito-converse}
\text{(i)} \quad \Gamma_x=\Gamma(1-\Gamma)\delta,\quad \text{and} \quad
\text{(ii)} \quad \Gamma_t+\Gamma_x(\mu_0+\Gamma \delta)
+\tfrac12\Gamma_{xx}=0,
\end{equation}
evaluated at $(t,x)=(t,X_t)$ for $(dt\otimes \widetilde{\mathbb P})$--almost every $(t,\omega)$ on $\llbracket 0,\tau_K\llbracket$. Take any $T>0$. An application of Tonelli's theorem yields that this holds for Lebesgue-almost all $t\in(0,T)$, $\widetilde{\mathbb{P}}$--a.s.~on $\{t<\tau_K\}$. $X$ has full support on the interior of $K$ by Lemma \ref{lem:stopped-full-support}, so continuity implies that this can be upgraded to pointwise on $(0,T)\times K$. Exhausting $J$ by compact intervals as in the proof of Proposition \ref{prop:likelihood-posterior} and gives
\eqref{eq:G-Ito-converse} pointwise on $(0,\infty)\times J$,
since $T>0$ was arbitrary. 
Dividing \eqref{eq:G-Ito-converse}(i)
by $\Gamma(1-\Gamma)$ yields $R_x(t,x)=\delta(x)$,
and hence there exists a $c\in C([0,\infty);\mathbb R)\cap C^{1}((0,\infty);\mathbb R)$ such that $R(t,x)=A(x)+c(t)$, where $A$ is given by
\eqref{eq:path-independent-L}. Moreover, from direct computation, we get $\Gamma_t=\Gamma(1-\Gamma)R_t$ and $\Gamma_{xx}
=\Gamma(1-\Gamma)\{\delta'+(1-2\Gamma)\delta^2\}$.
Substituting into \eqref{eq:G-Ito-converse}(ii)
and dividing by $\Gamma(1-\Gamma)$
gives $c'+\mu_0\delta
+\frac12\delta'+\frac12\delta^2=0$.
The first term depends only on $t$, while the remaining terms depend
only on $x$. Therefore, they are constant and there exists a $\kappa\in\mathbb{R}$ such that $c'(t)=\kappa/2$, which yields \eqref{eq:riccati}. Before showing the converse we prove the final statement about $L$ and $\Gamma$.
Since $\Gamma(0,x_0)=p$ and $A(x_0)=0$, we have
\(c(0)=\logit p\), which proves \eqref{eq:current-state-map} through the definition of $R$. By It\^o's formula on $A(X_{t\wedge\tau_K})$, \eqref{eq:likelihood-process}, \eqref{eq:riccati}, and $\mu_1^2-\mu_0^2=2\mu_0\delta+\delta^2$ we get \eqref{eq:path-independent-L}. 

Suppose now that \eqref{eq:riccati} holds. We can obtain \eqref{eq:path-independent-L} by repeating the argument above. Bayes' formula \eqref{eq:posterior-odds} then yields
\(
P_{t}
=
H\left(
\logit p+A(X_{t})
+\kappa t/2
\right),
\)
which proves that $P_t=\Gamma(t,X_t)$, and hence $P$ is a C-S posterior.
\end{proof}

The characterization in Theorem \ref{thm:recognition} can be recast as a Doob
$h$--transform condition on the \emph{null generator}
$\mathcal L_0f:=\mu_0f'+\tfrac12f''$ of $X$ under $\widetilde{\mathbb Q}^0$;
cf.\ Section \ref{sec:doob} and Remark \ref{rmk:doob_h_embedding}.

\begin{corollary}\label{cor:h-transform}
$P$ is a C-S posterior if and only if there exist $\kappa\in\mathbb R$ and
$h\in C^2(J;(0,\infty))$ with $\delta=h'/h$ such that
$\widehat h(t,x):=e^{\kappa t/2}h(x)$ is space-time harmonic for $\mathcal L_0$.
In this case, the local likelihood-ratio process is $L_t=\widehat h(t,X_t)/\widehat h(0,x_0)$ for $t<\zeta$.
\end{corollary}

\begin{proof}
By Theorem \ref{thm:recognition}, it suffices to show that \eqref{eq:riccati} holds
for some $\kappa$ if and only if such a pair $(\kappa,h)$ exists. On $J$, $\delta=h'/h$
holds if and only if $h=Ce^{A}$ for some $C>0$. For such $h$, $h'/h=\delta$ and
$h''/h=\delta'+\delta^2$, so $2\mathcal L_0h/h=\delta'+2\mu_0\delta+\delta^2$. Hence
\eqref{eq:riccati} is equivalent to $\mathcal L_0h=-\kappa h/2$; i.e.\ to $\widehat h$
being space-time harmonic. In this case \eqref{eq:path-independent-L} gives
$L_t=\exp\{A(X_t)+\kappa t/2\}=\widehat h(t,X_t)/\widehat h(0,x_0)$.
\end{proof}

Thus, $\widetilde{\mathbb Q}^1$ is (locally) the Doob
$\widehat h$--transform of $\widetilde{\mathbb Q}^0$.
In general, the Girsanov density
\eqref{eq:likelihood-process} is a functional of the entire observed path. Corollary \ref{cor:h-transform} and
\eqref{eq:posterior-odds} show that $P$ is a function of $(t,X_t)$
precisely when this functional collapses to an $h$--transform of this type.

\begin{remark}\label{rmk:scope}
Note that autonomy of $P$ does not imply the C-S property, which is \emph{not invariant}
under filtration-preserving transformations. For
instance, if $\delta\equiv\Delta\neq0$ and $\mu_0(x)=-x$ on $J=\mathbb R$,
\eqref{eq:filtering-equation} gives $dP_t=\Delta P_t(1-P_t)\,d\overline W_t$,
yet \eqref{eq:riccati} fails. Nonetheless, $P$ is a C-S posterior for the signal $S_t:=X_t-x_0+\int_0^tX_s\,ds$ which has drift $\Delta\theta$ and $\mathbb F^S=\mathbb F^X$.
\end{remark}

\subsection{Autonomous current-state posteriors}
Throughout this subsection we assume that $P$ is a C-S posterior, so that, by
Theorem \ref{thm:recognition}, \eqref{eq:riccati} holds and $P_t=\Gamma(t,X_t)$
with $\Gamma$ as in \eqref{eq:current-state-map}. Thus $(t,X_t)$ determines the
posterior, but the converse may fail, since the posterior need not retain all the
information about $X$. The following establishes that $X_t$ can be recovered from $(t,P_t)$ when $\delta$ has no roots.

\begin{lemma}\label{lem:invertible-current-state}
Suppose $\delta\not=0$ on $J$. Then, for each $t\geq0$, the map
$\Gamma(t,\cdot):J\to\Gamma(t,J)$ is strictly monotone and admits an inverse $f(t,q)$ for $q\in\Gamma(t,J)$. In particular, $X_t=f(t,P_t)$ for $t<\zeta$ and \eqref{eq:filtering-equation} becomes
\begin{equation}\label{eq:P_time_inhomogeneous}
  dP_t=P_t(1-P_t)\,\delta\bigl(f(t,P_t)\bigr)\,d\overline W_t,
  \qquad t<\zeta.
\end{equation}
\end{lemma}

\begin{proof}
Since $\delta$ is continuous and nonzero on the interval $J$, it has a constant sign. Hence $A$ is strictly monotone, and so is \(x\mapsto \Gamma(t,x)\), since \(H\) is strictly increasing. The inverse \(f(t,\cdot)\) therefore exists. Substituting
\(X_t=f(t,P_t)\) into \eqref{eq:filtering-equation} gives
\eqref{eq:P_time_inhomogeneous}.
\end{proof}

Although the underlying experiment is time-homogeneous, the coefficient in
\eqref{eq:P_time_inhomogeneous} is in general time-dependent, so the posterior
need not be an autonomous win-martingale.
By contrast, in the embedding of Theorem \ref{thm:posterior-embedding} the
posterior is a \emph{time-independent} transformation of the signal, $P=G(X)$.
The following theorem shows that this occurs precisely when $\kappa=0$, in which
case \eqref{eq:P_time_inhomogeneous} is time-homogeneous. On the other hand,
$P$ satisfies \eqref{eq:intro-target} before $\zeta$ for $\kappa\neq0$ precisely when $\delta$ is constant.

\begin{theorem}[Classification of autonomous C-S posteriors]\label{thm:classification}
\emph{(i)} There exists $\gamma\in C(J;(0,1))$ with $P_t=\gamma(X_t)$ for $t<\zeta$
if and only if $\kappa=0$. In this case, either $\delta\equiv0$ and $P\equiv p$, or
$\delta$ has no zeros, $\gamma$ is a $C^2$--diffeomorphism onto $\gamma(J)$ with
inverse $f$, and
\begin{equation}\label{eq:Pi-autonomous-zero-kappa}
  dP_t=P_t(1-P_t)\,\delta(f(P_t))\,d\overline W_t,\qquad t<\zeta.
\end{equation}
\emph{(ii)} Let $\kappa\neq0$. Then $P$ satisfies \eqref{eq:intro-target} for
$t<\zeta$ if and only if $\delta$ is constant. In this case $\delta\equiv\Delta\neq0$ and $dP_t=\Delta P_t(1-P_t)\,d\overline W_t$,  $t<\zeta$.
\end{theorem}
\begin{proof}
(i) If $\kappa=0$, then \eqref{eq:current-state-map} gives $P_t=\gamma(X_t)$ with
$\gamma:=H(\logit p+A(\cdot))$, and \eqref{eq:riccati} reads
$\delta'=-2\mu_0\delta-\delta^2$. Since the right-hand side is locally Lipschitz in $\delta$, Picard--Lindel\"of ensures uniqueness of the solution to this initial-value problem. As it admits the zero solution, either $\delta\equiv0$, whence $P\equiv p$ by
\eqref{eq:filtering-equation}, or $\delta$ has no zeros. In the latter case
$\gamma'=\gamma(1-\gamma)\delta$ has no zeros, so $\gamma$ is a
$C^2$--diffeomorphism onto $\gamma(J)$, and substituting $X_t=f(P_t)$ into
\eqref{eq:filtering-equation} gives \eqref{eq:Pi-autonomous-zero-kappa}.
Conversely, if $P_t=\gamma(X_t)$ with $\gamma\in C(J;(0,1))$, the same argument
as in the proof of Theorem \ref{thm:recognition} yields $\Gamma(t,x)=\gamma(x)$ on
$(0,\infty)\times J$. Since $H$ is injective, \eqref{eq:current-state-map} forces
$\kappa=0$.

(ii) If $\delta\equiv\Delta$, then $\Delta\neq0$, as \eqref{eq:riccati} would
otherwise give $\kappa=0$. Thus, \eqref{eq:filtering-equation}
has the stated dynamics and \eqref{eq:intro-target} holds with $\sigma(q)=|\Delta|q(1-q)$ and $B=\operatorname{sgn}(\Delta)\overline W$. Conversely,
suppose $dP_t=\sigma(P_t)\,dB_t$ for $t<\zeta$ and set
$\Psi(q):=\sigma(q)^2/(q(1-q))^2$. Comparing quadratic variations with
\eqref{eq:filtering-equation} and arguing as in the proof of Theorem
\ref{thm:recognition} gives $\Psi(\Gamma(t,x))=\delta(x)^2$ on $(0,\infty)\times J$.
By \eqref{eq:current-state-map}, $t\mapsto\Gamma(t,x)$ is strictly monotone and
converges, as $t\to\infty$, to $1$ if $\kappa>0$ and to $0$ if $\kappa<0$. Hence,
for $x_1,x_2\in J$, the ranges $I_i:=\Gamma((0,\infty),x_i)$ intersect, and any
$q\in I_1\cap I_2$ gives $\delta(x_1)^2=\Psi(q)=\delta(x_2)^2$. Thus $\delta^2$ is
constant, so by continuity on the interval $J$ so is $\delta$, and
$\delta\equiv\Delta\neq0$ as above.
\end{proof}

\begin{remark}\label{rmk:local}
Let $\eta:=\inf\{t\geq0:P_t\in\{0,1\}\}$. By \eqref{eq:posterior-odds},
$P_t\in(0,1)$ for $t<\zeta$, and $P$ is constant after $\zeta$, since the
absorbed signal reveals no further information. Hence $\zeta\leq\eta$ a.s.,
with $\eta=\zeta$ if and only if $P_\zeta\in\{0,1\}$ on $\{\zeta<\infty\}$.
As an autonomous win-martingale is absorbed only upon reaching $\{0,1\}$,
Corollary \ref{cor:classification} determines when the dynamics of
Theorem \ref{thm:classification} extend accordingly.
\end{remark}

\begin{corollary}\label{cor:classification}
$P$ is an autonomous win-martingale if and only if either \emph{(i)} $\kappa=0$ and
$\gamma(J)=(0,1)$, or \emph{(ii)} $\kappa\neq0$, $\delta$ is constant, and
$J=\mathbb R$.
\end{corollary}

\begin{proof}
If $\delta\equiv0$, then $P\equiv p$ has zero quadratic variation and
$\gamma(J)=\{p\}$, so both sides of (i) fail. Assume $\delta\not\equiv0$ and set
$B:=\operatorname{sgn}(\delta)\overline W$, extended past $\zeta$ if necessary.

(i) Let $\kappa=0$ and $I:=\gamma(J)$. If $I=(0,1)$, then
$\sigma(q):=q(1-q)|\delta(f(q))|$ lies in $C^1((0,1);(0,\infty))$ and
\eqref{eq:Pi-autonomous-zero-kappa} reads $dP_t=\sigma(P_t)\,dB_t$ for $t<\zeta$.
Moreover, $\gamma$ extends continuously to $\overline J$ with endpoint values
$\{0,1\}$, so $P_{\zeta-}\in\{0,1\}$ on $\{\zeta<\infty\}$. Optional sampling at
$\tau_{K_n}\uparrow\zeta$ and L\'evy's upward theorem give
$P_{\zeta-}=\widetilde{\mathbb E}\bigl[P_\zeta\mid\bigvee_n\mathcal F^X_{\tau_{K_n}}\bigr]$,
and since $P_{\zeta-}\in\{0,1\}$ while $P_\zeta\in[0,1]$, $P_\zeta=P_{\zeta-}$.
By Remark \ref{rmk:local}, $P$ is constant on $\llbracket\zeta,\infty\llbracket$,
so $P$ satisfies \eqref{eq:intro-target} with absorption at $\{0,1\}$. Conversely,
if $P$ satisfies \eqref{eq:intro-target}, Lemma \ref{lem:pi-well-defined} gives
$P_\infty\in\{0,1\}$ a.s.\ with both values of positive probability, while
continuity of $P$, $P_t\in I$ for $t<\zeta$, and Remark \ref{rmk:local} give
$P_t\in\overline I$ for all $t\geq0$; hence $\{0,1\}\subseteq\overline I$ and
$I=(0,1)$.

(ii) Let $\kappa\neq0$. If $P$ satisfies \eqref{eq:intro-target}, then
Theorem \ref{thm:classification} gives $\delta\equiv\Delta\neq0$, and
\eqref{eq:riccati} makes $\mu_0$ and $\mu_1$ constant, so under each
$\widetilde{\mathbb Q}^i$ the signal is a Brownian motion with constant drift
stopped at $\partial J$. If $J\neq\mathbb R$, then
$\widetilde{\mathbb P}(\zeta<\infty)>0$, since such a process reaches any finite
level with positive probability. On $\{\zeta<\infty\}$, non-explosion forces
$X_\zeta$ to be finite, so
$P_\infty=P_\zeta\in(0,1)$ by \eqref{eq:current-state-map}, continuity, and Remark
\ref{rmk:local}, contradicting $P_\infty\in\{0,1\}$ a.s. Hence $J=\mathbb R$.
Conversely, if $\delta\equiv\Delta$ and $J=\mathbb R$, then $\zeta=\infty$ a.s.\
and the dynamics are as in \eqref{eq:intro-target} with
$\sigma(q)=|\Delta|q(1-q)$.
\end{proof}

\color{black}

\begin{remark}\label{rmk:integral-criterion}
For $\kappa=0$ the condition $\gamma(J)=(0,1)$
holds if and only if
 $\int_\ell^{x_0}|\delta(u)|\,du=\int_{x_0}^{r}|\delta(u)|\,du=\infty$, which is a
state-space counterpart of the condition in Remark
\ref{rmk:consistency}.
\end{remark}

The $\kappa=0$ branch of Theorem \ref{thm:classification} can be restated via Doob
transforms by $\mathcal L_0$--harmonic functions, which are precisely the affine
functions of the null scale function $s_0$ from the proof of Lemma \ref{lem:signal-well-posed}. This suggests an inverse design. Given a null drift, the following corollary produces alternatives with time-independent autonomous posterior dynamics. For
$\kappa\neq0$ no such freedom exists as by Theorem \ref{thm:classification},
autonomy forces $\delta\equiv\Delta$, and then \eqref{eq:riccati} forces $\mu_0$
and $\mu_1$ to be constants.

\begin{corollary}\label{cor:inverse-design}
There exists $\gamma\in C(J;(0,1))$ with $P_t=\gamma(X_t)$ for $t<\zeta$ if and
only if $\delta=h'/h$ for a $\mathcal L_0$--harmonic
$h\in C^2(J;(0,\infty))$. Moreover, the nonvanishing solutions of
\eqref{eq:riccati} with $\kappa=0$ are $\delta=s_0'/(C+s_0)$, $C\in\mathbb R$,
on any subinterval of $J$ on which $C+s_0\neq0$.
\end{corollary}

\begin{proof}
By the identity $2\mathcal L_0h/h=\delta'+2\mu_0\delta+\delta^2$ (cf.~Corollary \ref{cor:h-transform}), such an $h$ exists if and only if \eqref{eq:riccati} holds with $\kappa=0$. The first claim then follows from Theorem \ref{thm:classification}. For the second, set $Y=1/\delta$. Then \eqref{eq:riccati} with $\kappa=0$ becomes $Y'-2\mu_0Y=1$. Since $s_0''=-2\mu_0s_0'$ this reads $(s_0'Y)'=s_0'$, whose solutions are $s_0'Y=C+s_0$; equivalently, $\delta=s_0'/(C+s_0)$.
\end{proof}

Taken together, the results complete the one-dimensional link between posterior martingale diffusions and binary diffusion experiments. On the other hand, extensions to higher dimensions with multiple hypotheses (cf.~\cite{bayraktar-wang2025, ekstrom2022multi}) no longer have access to the same rich scalar theory and so appear to require different geometric tools.

\section{Examples}\label{sec:examples}

\begin{table}[t]
    \centering
    \caption{Autonomous win-martingales and their associated binary experiments.}
    \label{tab:examples}
    \footnotesize
    \renewcommand{\arraystretch}{1.5}
    \setlength{\tabcolsep}{3.5pt}
    \begin{tabular}{@{}ccccc@{}}
        \hline
        $\sigma(p)$
        &
        $F(p)$
        &
        $\mu_0(x)$
        &
        $\mu_1(x)$
        &
        $\mu_{\mathrm{mix}}(x)$
        \\
        \hline

        $\displaystyle p(1-p)$
        &
        $\displaystyle \logit(p)$%
        &
        $\displaystyle -1/2$
        &
        $\displaystyle 1/2$
        &
        $\displaystyle \tanh\left(x/2\right)/2$
        \\[1.5ex]

        $\displaystyle 1$
        &
        $\displaystyle p-\frac12$
        &
        $\displaystyle -\frac{1}{1/2-x}$
        &
        $\displaystyle \frac{1}{x+1/2}$
        &
        $\displaystyle 0$
        \\[2ex]

        $\displaystyle \sqrt{p(1-p)}$
        &
        $\displaystyle 2\arcsin\sqrt p-\frac{\pi}{2}$
        &
        $\displaystyle -\sec x-\frac12\tan x$
        &
        $\displaystyle \sec x-\frac12\tan x$
        &
        $\displaystyle \frac12\tan x$
        \\[2ex]

        $\displaystyle \frac{\sin(\pi p)}{\pi}$
        &
        $\displaystyle \log\tan\left(\frac{\pi p}{2}\right)$
        &
        $\displaystyle
        \frac12\tanh x
        -\frac{\operatorname{sech}x}{\arccos(\tanh x)}$
        &
        $\displaystyle
        \frac12\tanh x
        +\frac{\operatorname{sech}x}{\arccos(-\tanh x)}$
        &
        $\displaystyle \frac12\tanh x$
        \\[2.5ex]

        $\displaystyle (1-p)^2$
        &
        $\displaystyle \frac{p}{1-p}$
        &
        $\displaystyle 0$
        &
        $\displaystyle \frac1x$
        &
        $\displaystyle \frac{1}{1+x}$
        \\[2ex]

        $\displaystyle \varphi\bigl(\Phi^{-1}(p)\bigr)$
        &
        $\displaystyle \Phi^{-1}(p)$
        &
        $\displaystyle
        \frac{x}{2}
        -\frac{\varphi(x)}{1-\Phi(x)}$
        &
        $\displaystyle
        \frac{x}{2}
        +\frac{\varphi(x)}{\Phi(x)}$
        &
        $\displaystyle \frac{x}{2}$
        \\
        \hline
    \end{tabular}
\end{table}

We now illustrate the embedding and its converse in the examples of Table \ref{tab:examples}, where $\varphi$ and $\Phi$ are the standard normal density and distribution function. Each row can be read either as a given $\sigma$ generating an experiment via Theorem \ref{thm:posterior-embedding}, or as a prescribed $\mu_0$ generating an autonomous win-martingale via Corollaries \ref{cor:classification} and \ref{cor:inverse-design}. 
The final column records the drift of $X$ under $\widetilde{\mathbb P}$, namely $\mu_{\mathrm{mix}}(x)=\widehat\mu\bigl(x,\gamma(x)\bigr) =-\tfrac12\sigma'(G(x))$, in accordance with \eqref{eq:Y-under-P}, since $\operatorname{Law}_{\widetilde{\mathbb P}}(X) =\operatorname{Law}_{\mathbb P}(F(\Pi))$.

The classical sequential testing martingale, $\sigma(p)=p(1-p)$, arises from testing between Brownian motions with drifts $\mp1/2$ \cite{shiryaev1967two}. The textbook presentation with drifts $(0,1)$ differs by the shift $X\mapsto X+t/2$, which changes $\kappa$ from $0$ to $-1$ but not the underlying experiment; cf.\ Remark \ref{rmk:scope}.

The case $\sigma\equiv1$ gives a Brownian motion absorbed at $0$ and $1$ in almost surely finite time. Under $\widetilde{\mathbb Q}^0$ and $\widetilde{\mathbb Q}^1$, respectively, $1/2-X$ and $X+1/2$ are three-dimensional Bessel processes. Hence the boundary corresponding to the incorrect hypothesis is unattainable, while the correct one is reached in finite time, revealing $\theta$ and absorbing $P$, when $X$ exits $J=(-1/2,1/2)$. Under the mixture law, $X$ is a Brownian motion up to $\zeta$.

The Wright--Fisher win-martingale, $\sigma(p)=\sqrt{p(1-p)}$, is likewise absorbed at $0$ or $1$ in finite time. The associated experiment has the bounded state space $J=(-\pi/2,\pi/2)$, whose boundary the observation reaches and is absorbed at under the mixture law. The fourth row gives the time-homogeneous Aldous martingale. The identification with the most exciting game of \cite{backhoff-beiglbock} follows from the time change $s=1-e^{-t}$.

The fifth row illustrates Corollary \ref{cor:inverse-design}: for $\mu_0\equiv0$ on $J=(0,\infty)$, $s_0(x)=x-x_0$ and the choice $C=x_0$ give $\delta(x)=1/x$. The null is thus a Brownian motion absorbed at $0$ and the alternative a three-dimensional Bessel process. Here $\sigma(p)=(1-p)^2$ and $f=F$, the Lamperti constant in \eqref{eq:F-def} being $1$. The posterior reaches $0$ in finite time on $\{\theta=0\}$ but approaches $1$ only asymptotically on $\{\theta=1\}$. As a result, certainty of the alternative is never attained in finite time.

The final row constructs the experiment from the mixture law and illustrates the Doob $h$--transform connection. Under $\widetilde{\mathbb P}$, $X$ is an unstable Ornstein--Uhlenbeck process. Since
$e^{-t/2}X_t=x_0+\int_0^te^{-s/2}\,d\overline W_s$ converges almost surely to a nondegenerate Gaussian limit, $X_t\to\pm\infty$ a.s.\ and $\widetilde{\mathbb P}\bigl(\lim_{t\to\infty}X_t=\infty\mid\mathcal F^X_t\bigr)=\Phi(X_t)$. The functions $h_1:=\Phi$ and $h_0:=1-\Phi$ are harmonic for $\mathcal Lf:=\mu_{\mathrm{mix}}f'+\tfrac12f''$, so by \eqref{eq:doob-conditioning} their Doob transforms are the laws of $X$ conditioned on $\{\lim_{t\to\infty}X_t=\pm\infty\}$. Taking these as $\widetilde{\mathbb Q}^1$ and $\widetilde{\mathbb Q}^0$, with prior $p=\Phi(x_0)$, mirrors \eqref{eq:embedding-terminal-densities} and yields the drifts $\mu_i=x/2+h_i'/h_i$ of Table \ref{tab:examples}, together with $d\widetilde{\mathbb Q}^1/d\widetilde{\mathbb Q}^0|_{\mathcal F^X_t} =h(X_t)/h(x_0)$ for $h:=h_1/h_0$. Since $\mathcal L_0f=\mathcal L(h_0f)/h_0$ and $\mathcal Lh_1=0$, $h$ is $\mathcal L_0$--harmonic with $\delta=h'/h$, so Corollaries \ref{cor:inverse-design} and \ref{cor:classification} show that the posterior is the autonomous win-martingale with $\sigma(p)=\varphi(\Phi^{-1}(p))$ and $\gamma=\Phi$. This becomes the Bass martingale \cite{backhoff2025bass} after the same deterministic time change $s=1-e^{-t}$.

In sequential testing, stopping rules are often given by the first exit of the posterior from an interval. That is, reject the null when $P_t$ hits $\beta$, and reject the alternative when $P_t$ hits $\alpha$, for $0<\alpha<p<\beta<1$. When the posterior is an autonomous win-martingale, the corresponding error probabilities and expected decision time admit particularly simple expressions. The following result collects these formulas for a general posterior-threshold rule.

\begin{corollary}[Posterior thresholds]\label{cor:thresholds}
    Let $0<\alpha<p<\beta<1$ and $\tau_{\alpha,\beta}$ be the first exit time
    \(
        \tau_{\alpha,\beta}
        =\inf\{t\geq 0:P_t\notin(\alpha,\beta)\},
    \)
    where $P$ is the posterior for a binary experiment. Assume $P$ is an autonomous win-martingale with $P_0=p$. Then $\tau_{\alpha,\beta}<\infty$ a.s.~and
    \begin{equation}\label{eq:threshold-hitting}
        \widetilde{\mathbb P}(P_{\tau_{\alpha,\beta}}=\beta)
        =\frac{p-\alpha}{\beta-\alpha},
        \qquad
        \widetilde{\mathbb P}(P_{\tau_{\alpha,\beta}}=\alpha)
        =\frac{\beta-p}{\beta-\alpha}.
    \end{equation}
    If the decision is $1$ at $\beta$ and $0$ at $\alpha$, then
    \begin{align}
        \widetilde{\mathbb Q}^0(\textnormal{decide }1)
        &=\frac{1-\beta}{1-p}\frac{p-\alpha}{\beta-\alpha},
        \qquad
        \widetilde{\mathbb Q}^1(\textnormal{decide }0)
        =\frac{\alpha}{p}\frac{\beta-p}{\beta-\alpha}.
    \label{eq:threshold-errors}
    \end{align}
    The expected decision time under the marginal law is
    \begin{align}
        \widetilde{\mathbb E}[\tau_{\alpha,\beta}]
        =2\bigg[&
        \frac{\beta-p}{\beta-\alpha}
        \int_\alpha^p\frac{y-\alpha}{\sigma^2(y)}\,dy
        +\frac{p-\alpha}{\beta-\alpha}
        \int_p^\beta\frac{\beta-y}{\sigma^2(y)}\,dy
        \bigg].
    \label{eq:threshold-time}
    \end{align}
\end{corollary}

\begin{proof}
    The almost-sure finiteness of $\tau_{\alpha,\beta}$ and equations \eqref{eq:threshold-hitting} and \eqref{eq:threshold-time} follow from \cite[Section 5.5.C]{karatzas-shreve}. Since $p=\widetilde{\mathbb P}(\theta=1)$ and $\widetilde{\mathbb Q}^1=\widetilde{\mathbb P}(\cdot\mid\theta=1)$ we have $p\widetilde{\mathbb Q}^1(A)=\widetilde{\mathbb E}[\mathbb 1_{A}\theta]$ for $A\in\widetilde{\mathcal{F}}$. Thus, by the tower property we have
    \begin{equation*}
        p\widetilde{\mathbb Q}^1(P_{\tau_{\alpha,\beta}}=\alpha)=\widetilde{\mathbb E}[\mathbb 1_{\{P_{\tau_{\alpha,\beta}}=\alpha\}}\theta]=\widetilde{\mathbb E}[\mathbb 1_{\{P_{\tau_{\alpha,\beta}}=\alpha\}}P_{\tau_{\alpha,\beta}}]=\alpha\widetilde{\mathbb P}(P_{\tau_{\alpha,\beta}}=\alpha),
    \end{equation*}
    which shows the second part of \eqref{eq:threshold-errors}. The first part is similar.
\end{proof}

For fixed threshold strategies, \eqref{eq:threshold-errors} shows that the type I and type II error probabilities depend on the thresholds, but not directly on the volatility of the posterior. However, the optimal thresholds generally depend on the particular diffusion. Consequently, both the mean decision time and the distribution of the resulting decision may depend on the posterior volatility when the thresholds are chosen optimally. The associated sequential inference and optimal stopping problems are studied further in forthcoming work by the authors.

\section{Conclusion}

Taken together, the results complete the one-dimensional link between posterior martingale diffusions and binary diffusion experiments. The embedding realizes any autonomous win-martingale by placing an absorbed diffusion in Lamperti coordinates and conditioning on its terminal outcome. The converse identifies when the posterior is a Markovian function of the current time and observation through a single compatibility condition, expressed either as a Riccati equation for the drift gap or as a space-time harmonic representation of the likelihood ratio. The spatially harmonic case recovers precisely the family produced by the embedding while outside it, autonomy forces both signal drifts to be constant and returns the classical binary-testing filter. Extensions to higher dimensions with multiple hypotheses (cf.~\cite{bayraktar-wang2025, caffarelli1981sequential, ekstrom2022multi}) no longer have access to the same rich scalar theory and so appear to require different geometric tools.

\bibliographystyle{amsplain}
\bibliography{refs}